\documentclass[12pt]{amsart}

\usepackage[left=3cm,top=3.0cm,right=3cm,bottom=2.6cm]{geometry}

\usepackage[ansinew]{inputenc}
\usepackage{amscd,amsfonts,amsmath,amssymb,amsthm,enumerate,epsfig,float,graphics,graphicx,pst-all,yhmath}

\newcommand{\inte}{\operatorname*{int}}

\newcommand{\bd}{\operatorname*{bd}}

\newtheorem{lemma}{Lemma}
\newtheorem{proposition}{Proposition}
\newtheorem{theorem}{Theorem}

\newtheorem{conjecture}{Conjecture}

\title [A characterization of the sphere]{A characterization of the sphere in terms of the stereographic projection}
\author{Efr\'en Morales-Amaya}
\address{Facultad de Matem\'aticas-Acapulco,
Universidad Aut\'onoma de Guerrero, M\'exico}
\email{emoralesamaya@gmail.com}
\thanks{Dedicated to Peter Gruber.}   
\begin{document}

\maketitle
          \begin{abstract}          
    Let $K$ be a convex body in the 3-dimensional Euclidian space $\mathbb{E}^3$ and 
      let $N,S\in \bd K$, $N\not=S$. Suppose that the support  plane $\Pi_S$ of $K$ at $S$ is unique.   
          For every point $x\in \bd K\backslash \{N\}$ we define the \textit{stereographic 
          projection} $\Psi:\bd K\backslash \{N\} \rightarrow \Pi_S$ of $x$ onto $\Pi_S$ as the point 
          $y:=L(N,x)\cap \Pi_S$. 
          
          It is a well known property of the sphere $\mathbb{S}^2$ in $\mathbb{E}^3$  that the 
          stereographic projection maps circles onto circles (see \cite{Hilbert} pag. 248). In this work 
          we investigate what geometric elements determines that this property is fulfilled.

          Here we demonstrate that the following two properties of a convex body 
          $K\subset \mathbb{E}^3$ in terms of the stereographic projection characterize the sphere in 
       $\mathbb{E}^3$:
       \begin{enumerate}           
     \item The cones defined by the sections of $K$ and the point $N$ are axially symmetric (that 
     is, they are invariant under a rotation by an angle of $\pi$).
     \item given a section $K_\Gamma$ of $K$, 
      the rotation that leaves the cone defined by $K_\Gamma$ and $N$ invariant is such that it 
      maps $K_\Gamma$ into a homothetic figure to $\Psi(K_\Gamma)$ by a homothety with center 
      of homothety at $N$.
      \end{enumerate}
      An important element in the proof of the main theorem of this work is a cha\-rac\-te\-ri\-za\-tion 
      of the circle based on a geometric property, which will be called the stereographic property. 
      It is worth highlighting that the stereographic projection  defined on the sphere maps circles onto circles  
      is intimately linked to the conditions (1) and (2) and the stereographic property of the circle. 
               \end{abstract}
          \section{Introduction}      
          Let $W\subset \mathbb{E}^n $ be a bounded closed convex set. Given a point 
         $x \in \mathbb{E}^n \backslash W$ we 
         denote the cone generated by $W$ with apex $x$ by $S_x(W)$, that is, 
         $S_x(W) := \{x + \lambda(y - x) : y \in W, \lambda \geq  0\}$ and by $C_x(W)$ the boundary of 
         $S_x(W)$.
         
         Let $K\subset \mathbb{E}^3$ be a convex body  and let $x\in \bd K$. For every plane $\Gamma$ 
         such that $\Gamma\cap \inte K\not= \emptyset$ and $x\notin \Gamma$, let 
           $K_{\Gamma}:=\Gamma\cap K$. The cone $C_x(K_\Gamma)$ will be called a \textit{inscribed cone of the body $K$}. 
           
        Below we present two theorems that illustrate an interesting property 
        of the sphere and an ellipsoid of revolution in terms of inscribed cones.   
           
       \begin{theorem}\label{proyeesfera}
        All the inscribed cones of the sphere $\mathbb{S}^2$ are axially symmetric.
        \end{theorem}
       \begin{theorem}\label{proyeeli}
      Let $E\subset \mathbb{E}^3$ be a solid ellipsoid of revolution, let $L(N,S)$ be the axis of 
      $E$ and let $N,S\in \bd E$. Then all the inscribed cones of the ellipsoid $ E$ with apex at $N$ are 
      axially symmetric, i.e., for all plane $\Gamma$ such that $\Gamma\cap \inte K\not= \emptyset$ 
      and $N\notin \Gamma$, there exists a rotation $R^{\Gamma} $ with property that
       \begin{eqnarray}\label{shorcito}
       R^{\Gamma} (C_N(E_{\Gamma}))=C_N(E_{\Gamma}),
       \end{eqnarray}
       where $E_{\Gamma}:=\Gamma\cap E$.
      \end{theorem}

       The following result is a characterization of the sphere in terms of 
       inscribed cones
          \begin{theorem}\label{proyecaracte}
        Let $K\subset \mathbb{E}^3$ be a convex body. If all the inscribed cones of 
        the body $K$ are axially symmetric, then $K$ is an sphere.
        \end{theorem}

       Let $K\subset \mathbb{E}^3$ be a convex body and let $N,S\in \bd K$, $N\not=S$. 
          Suppose that the support  plane of $K$ at $S$ is unique. We denote such plane as 
          $\Pi_S$. For every point $x\in \bd K\backslash \{N\}$ we define the \textit{stereographic 
          projection} $\Psi:\bd K\backslash \{N\} \rightarrow \Pi_S$ of $x$ onto $\Pi_S$ as the point 
          $y:=L(N,x)\cap \Pi_S$, where $L(N,x)$ denotes the line defined by $N$ 
          and $x$. 
        Let $B^3:=\{(x,y,z)\in  \mathbb{R}^3: \|(x,y,z)\|\leq 1\}$ be the ball and let $\mathbb{S}^2:=\bd B^3$ 
        be the sphere.   
        
        Next, we present a well-known property of the sphere (see \cite{Hilbert} 
        pag. 248). In particular, we 
        present this result because we are going to give a proof that employs 
        the ideas used in the proofs of the other theorems in this work.
        \begin{proposition}\label{proyecirculo}
        Let $N,S\in \mathbb{S}^2$ such that the line segment $[x,y]$ is a diameter of 
       $\mathbb{S}^2$. We defined the stereographic projection 
       $\Psi:\mathbb{S}^2\backslash \{N\} \rightarrow \Pi_S$ with respect to 
       $\mathbb{S}^2, N, S, \Pi_S$. Then 
       \[
       \Psi \textrm{ }\textrm{  maps } \textrm{ }\textrm{ circles }\textrm{  onto }\textrm{ circles,}
       \] 
       i.e., for all plane $\Gamma$ such that 
       $\Gamma \cap \inte B^3\not=\emptyset$ and $N\notin \Gamma$ the set $\Psi(S_\Gamma)$ is a circle, 
       where $S_\Gamma:=\Gamma \cap \mathbb{S}^2$
         \end{proposition}

         The main result of this work is the following theorem which is a characterization of the sphere in 
         terms of the stereographic projection.
          
       \begin{theorem}\label{proyebody}
       Let $K\subset \mathbb{E}^3$ be a convex body  and let $N,S\in \bd K$, $N\not=S$. Suppose that the 
        {stereographic projection} is well defined for it and, for all plane $\Gamma$ such that 
       $\Gamma\cap \inte K\not= \emptyset$, $N\notin \Gamma$ and $\Gamma$ no parallel to 
       $\Pi_S$, there exists a rotation 
       $R^{\Gamma} $, by the angle 
        $\pi$ around the axis $L_{\Gamma}$, with property that
       \begin{eqnarray}\label{zapatito}
       R^{\Gamma} (C_N(K_{\Gamma}))=C_N(K_{\Gamma})
       \end{eqnarray}
       and 
       \begin{eqnarray}\label{playerita}
        \Psi (K_{\Gamma})=R^{\Gamma}(K'_{\Gamma}),
       \end{eqnarray}
        where $K'_{\Gamma}$ is a homothetic copy of $K_{\Gamma}$ by an homothety with center 
        at $N$. Then $K$ is a sphere.
       \end{theorem}   
       The property (\ref{zapatito}) means that the cone $C_N(K_{\Gamma})$ is invariant of 
       $R^{\Gamma}$. Notice that $K_{\Gamma}$, $K'_{\Gamma}$ and $\Psi (K_{\Gamma})$ 
       are sections of the cone $C_N(K_{\Gamma})$.  
        
        There are multiple investigations where geometric problems are studied to determine convex 
        bodies from cones with some symmetry or special geometric characteristic associated with such 
        body. See for example \cite{gjmc1}, \cite{gjmc2}, \cite{papiefen}, \cite{egj}, \cite{Amaya}.
          
       We propose the following conjecture.             
      \begin{conjecture}\label{conjeeli}
      Let $K\subset \mathbb{E}^3$ be a convex body and let $N \in K$. If, for all plane $\Gamma$ such that 
      $\Gamma\cap \inte K\not= \emptyset$ and $N\notin \Gamma$, 
      there exists a rotation $R^{\Gamma} $ with property that
       \begin{eqnarray}\label{camita}
       R^{\Gamma} (C_N(K_{\Gamma}))=C_N(K_{\Gamma}),
       \end{eqnarray}
       where $K_{\Gamma}:=\Gamma\cap K$, then $K$ is an ellipsoid of revolution.
      \end{conjecture}      
      
       Let $x\in \bd K$. The point $x$ is said to be a \textit{regular boundary point} of $K$ if there is only one support plane 
       of $K$ at $x$. Finally we give a progress in relation to Conjecture \ref{conjeeli}.      
             
      \begin{theorem}\label{teoeli2}
      Let $K\subset \mathbb{E}^3$ be a convex body and let $N \in K$ be a regular boundary point of $K$. 
      If all the inscribed cones of the body $K$ with apexes at $N$ are axially symmetric then $K$ is a body of revolution.
      \end{theorem}      

      The proof of the Theorem \ref{proyebody} consists of demonstrating:
      \begin{enumerate}
       \item the axes of cones defined by sections parallel to a direction parallel to $\Pi_S$ are coplanar (Lemma \ref{prerevolution}).
       \item  the solid $K$ is a solid of revolution with axis the line $L(N,S)$ (Lemma \ref{revolution}).
       \item  the sections of $K$ with planes containing $L(N,S)$ have the stereographic property (see the definition at the Section \ref{stereoproperty}) with 
       respect to $N,S$ and $L(N,S)$ and, therefore, are circles with dia\-me\-ter $[N,S]$ (Lemma \ref{seccionestereo}).
       \end{enumerate}  
       
      \section{Definitions and notation}\label{defs}
       Let $\mathbb{E}^{d}$ be the Euclidean space of dimension $d$ endowed with the usual scalar 
       product $\langle \cdot, \cdot\rangle : \mathbb{E}^{d} \times \mathbb{E}^{d} \rightarrow \mathbb{R}$. 
       Let $\mathbb{S}^{d-1}=\{x\in \mathbb{E}^{d}: \|x\| = 1\}$ be the unit sphere in $\mathbb{E}^{d}$. 
       Let $x, y \in \mathbb{R}^n$, we denote by $L(x, y)$ the line through $x$ and $y$, and by $[x, y]$ the line 
       segment connecting them.

        A {\it convex body} $K\subset {\mathbb E^d}$, $d\ge 2$,   is a convex compact set with non-empty interior. 
        A \textit{convex hypersurface} is the boundary of a convex body $K$ in $\mathbb{E}^{d}$ and it will be 
        denoted by $\bd K$. We will denote by $\inte K$ the set $K\backslash \bd K$.
        A {\it chord} of a convex body $K$ is any line segment $[x, y]$ in $K$ 
	such that $x,y\in \bd K$. An excellent book 
        where you can consult the basic concepts and results of convexity is \cite{Constantwidth}.
        \begin{figure}[H]
    \centering
     \includegraphics [width=.8\textwidth]{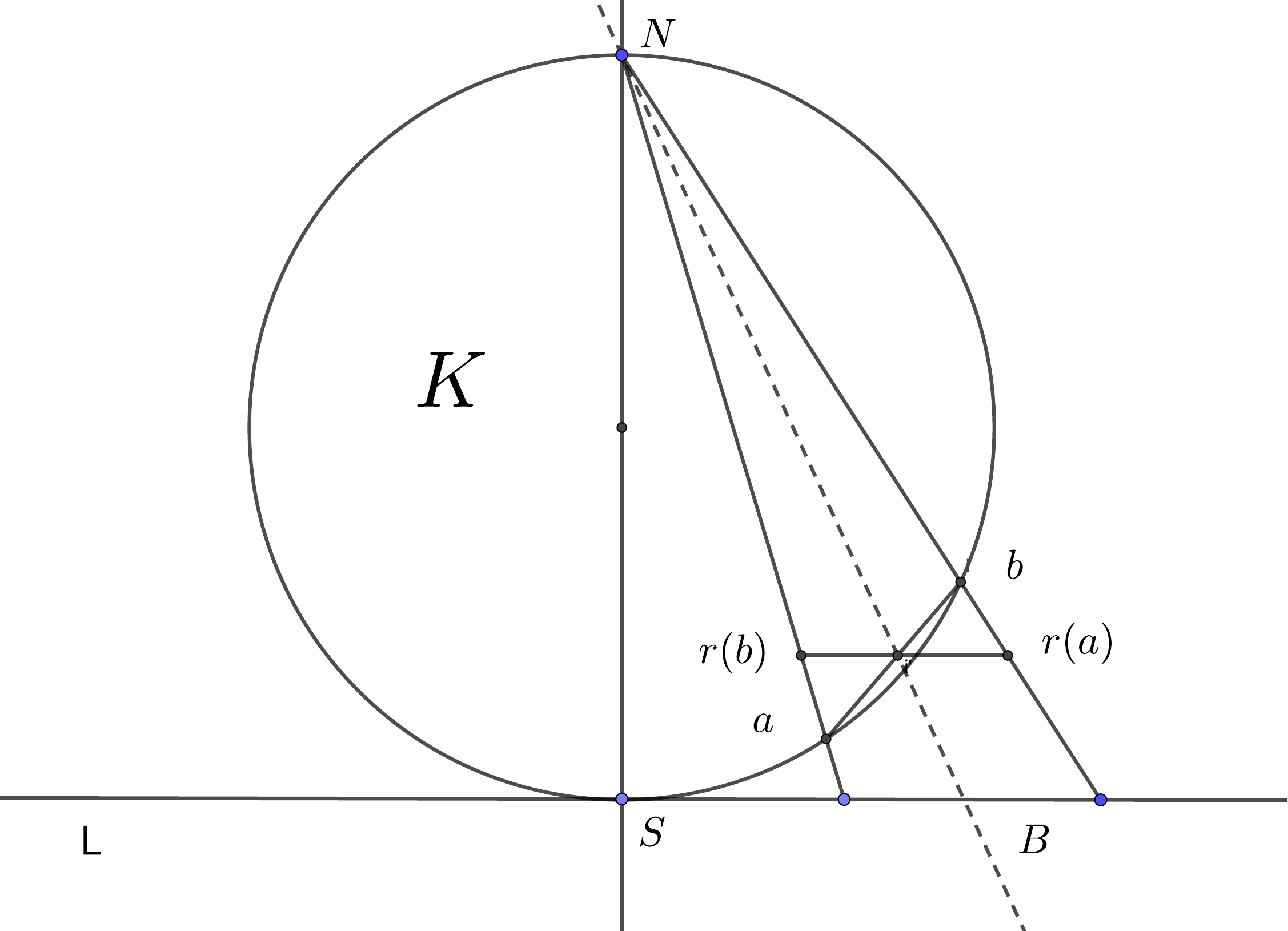}
    \caption{The Stereographic Property of a convex figure.}
    \label{nina}
     \end{figure}
     \begin{figure}[H]
      \centering
       \includegraphics [width=.8\textwidth]{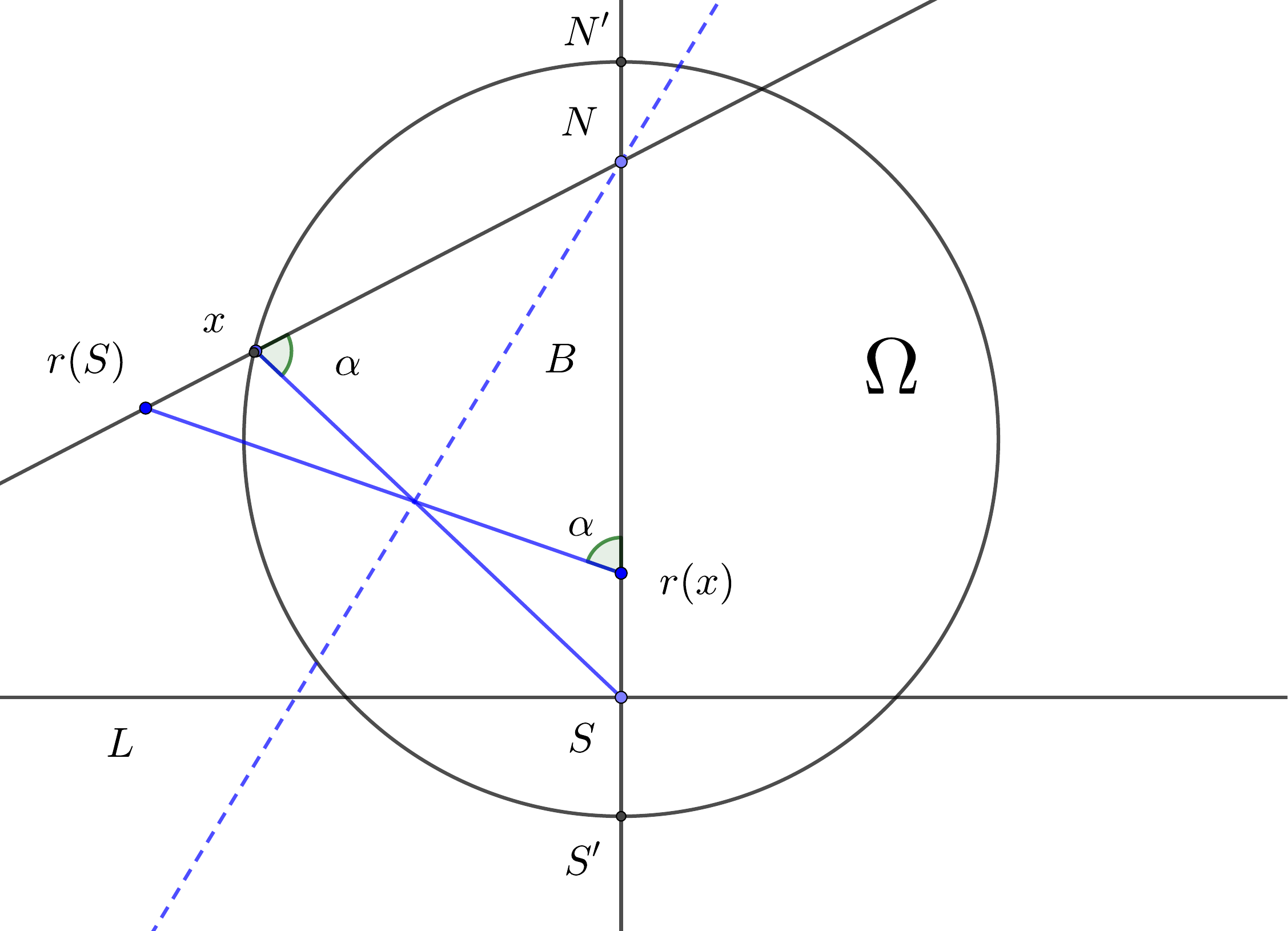}
       \caption{The Stereographic Property characterizes the circle. Case $N\not=N', S\not=S'$.}
        \label{delfin}
     \end{figure}     
        Let $K \subset \mathbb{E}^{d}$ be a convex body, let $L\subset \mathbb{E}^{3}$ be a line and let $\Pi$ be 
        a  hyperplane.  We denote by $R_{L}:\mathbb{E}^{3} \rightarrow \mathbb{E}^{3}$ the rotation by the angle 
        $\pi$ around the axis $L$ and by $S_{\Pi}: \mathbb{E}^{3}\rightarrow \mathbb{E}^{3}$ the reflection with 
        respect to $\Pi$. A line $L\subset \mathbb{E}^{3}$ is said to be an 
        {\it axis of symmetry} if   
        \[
        R_L (K)=K.
        \] 
        The hyperplane $\Pi$ is said to be a \textit{hyperplane of symmetry} of $K$ if 
         \[
         S_{\Pi}(K)=K.
         \]                     
                \section{Characterization of the Circle.}\label{stereoproperty}
        Let $K\subset \mathbb{E}^2$ be a convex body and let $N,S\in \bd K$, $N\not=S$. Suppose that the 
        support line of $K$ at $S$ is unique. We denote such line as $L$. For every chord $[a,b]$ of $K$ 
        we denote by $B$ the bisectriz of the angle defined by the lines $L(a,N)$ 
        y $L(b,N)$ and by $r(a),r(b)$ the reflections of $a,b$ with respect to $B$ (see fig. \ref{nina}). 
        The convex body $K$ possesses the \textit{stereographic property} with respect to the points $N,S$ 
        and the line $L$ if, for every chord $[a,b]$ of  
        $K$, the chord $[r(a),r(b)]$ is parallel to $L$.
                     
             Let $D^2:=\{(x,y)\in  \mathbb{R}^2: \|(x,y)\|\leq 1\}$ be the disc and let $\Omega:=\bd D^2$ be the circle.   
        \begin{theorem}\label{estereo}
         Let $K\subset  \mathbb{R}^2 $ be a convex body.  $K$ possesses the stereographic property with respect 
         to the points $N,S$ and the line $L$ if and only 
         if $K$ is a circle with diameter $[N,S]$. 
         \end{theorem}
          
          \begin{proof}It is easy to prove that if $K$ is a circle with diameter $[N,S]$, then $K$ possesses the stereographic 
         property with respect to the points $N,S$ and the line $L$. Thus we will 
         only consider the opposite direction, i.e., we will assume that $K$ possesses the stereographic 
         property with respect to the points $N,S$ and the line $L$ and we will prove that $K$ is a circle. 
                  
         \textbf{Claim.} The convex body $K$ is symmetric with respect to the line $L(N,S)$. 
         
          \textit{Proof}. Let $x\in \bd K$. Let $y\in \bd K$ such that the line $L(x,y)$ is parallel to $L$. By the 
          stereographic property, the line $L(r(x),r(y))$ is parallel to $L$. The only possibility is that the bisectriz $B$ of the 
          angle defined by the lines $L(N,x)$ and $L(N,y)$ is perpendicular to $L$. It follows that $y=r(x)$, $x=r(y)$ and the 
          mid-point of $[x,y]$ is in $B$. By virtue of the arbitrariness of $x\in K$, it follows that $L(N,S)$ is the perpendicular to $L$
          passing through $N$ and is the line of symmetry of $K$ (notice that $S=L\cap \bd K$, otherwise, it is easy to see that 
          it would contradict the stereographic property of $K$).
          $\square$
          
          Without loss of generality we may assume that the circumcircle of $K$ is $\Omega$. By virtue of the unicity of 
          the circumsphere $\Omega$, the center $c$ of $\Omega$ is in the line of symmetry $L(N,S)$ of 
         $K$ (otherwise, the image of $\Omega$ by the reflection in the line $L(N,S)$ it would be a circle of minimus ratio, 
         different than $\Omega$, which contains $K$, since $K$ would be invariant by such reflection). Let 
         \[
         \{N',S'\}:=L(N,S)\cap \Omega.
         \]  
         First we consider the case $N\not=N'$ and $S\not=S'$. There exists $x\in \bd K \cap \Omega$, $N\not=x\not=S$. 
         Since $K$ possesses the stereographic property with respect to the points $N,S$ and the line $L$,  the line segment 
         $[r(x),r(S)]$ is parallel to $L$, i.e., the angle $\alpha:=Nr(x)r(S)$ is equal to $\frac{\pi}{2}$. Since the triangles 
         $Nr(x)r(S)$ and $NxS$ are congruent the angle $NxS$ is equal to $\alpha$ (see fig. \ref{delfin}). 
         Thus $x\in \Omega'$, where $\Omega'$ is the circle with diameter $[N,S]$. Because $c\in L(N,S)$ and 
         $N\not=N', S\not=S'$ it follows that $\Omega'\subset \inte D^2$. Therefore $x\in \inte D^2$ but contradicts the choice of 
         $x$.        
         
         The case $N=N', S\not=S'$ and the case $N\not=N', S =S'$ can be considered in analogous way.
             
          Now we suppose that $N=N', S=S'$. Let $y\in \bd K$, $y\not=N,S$. Let $x:=L(N,y)$. Suppose that $x\not=y$. Then 
          $[r(y),r(S)]\not=[r(x),r(S)]$ (see fig. \ref{bebe}). Since $K$ possesses the stereographic property with respect to 
          the points $N,S$ and the line $L$, the segment $[r(y),r(S)]$ is parallel to $L$. On the other hand, since 
         $\Omega$ possesses the stereographic property with respect to the points $N,S$ 
        and the line $L$, the segment $[r(x),r(S)]$ is parallel to $L$. Therefore $[r(y),r(S)]$ and 
         $[r(x),r(S)]$ are parallel. However 
         \[
         r(S)=[r(y),r(S)]\cap  [r(x),r(S)].
         \]
         This contradiction shows that 
         $x=y$. Hence $\Omega=\bd K$. 
               \end{proof}
         \begin{figure}[H]
      \centering
       \includegraphics [width=.8\textwidth]{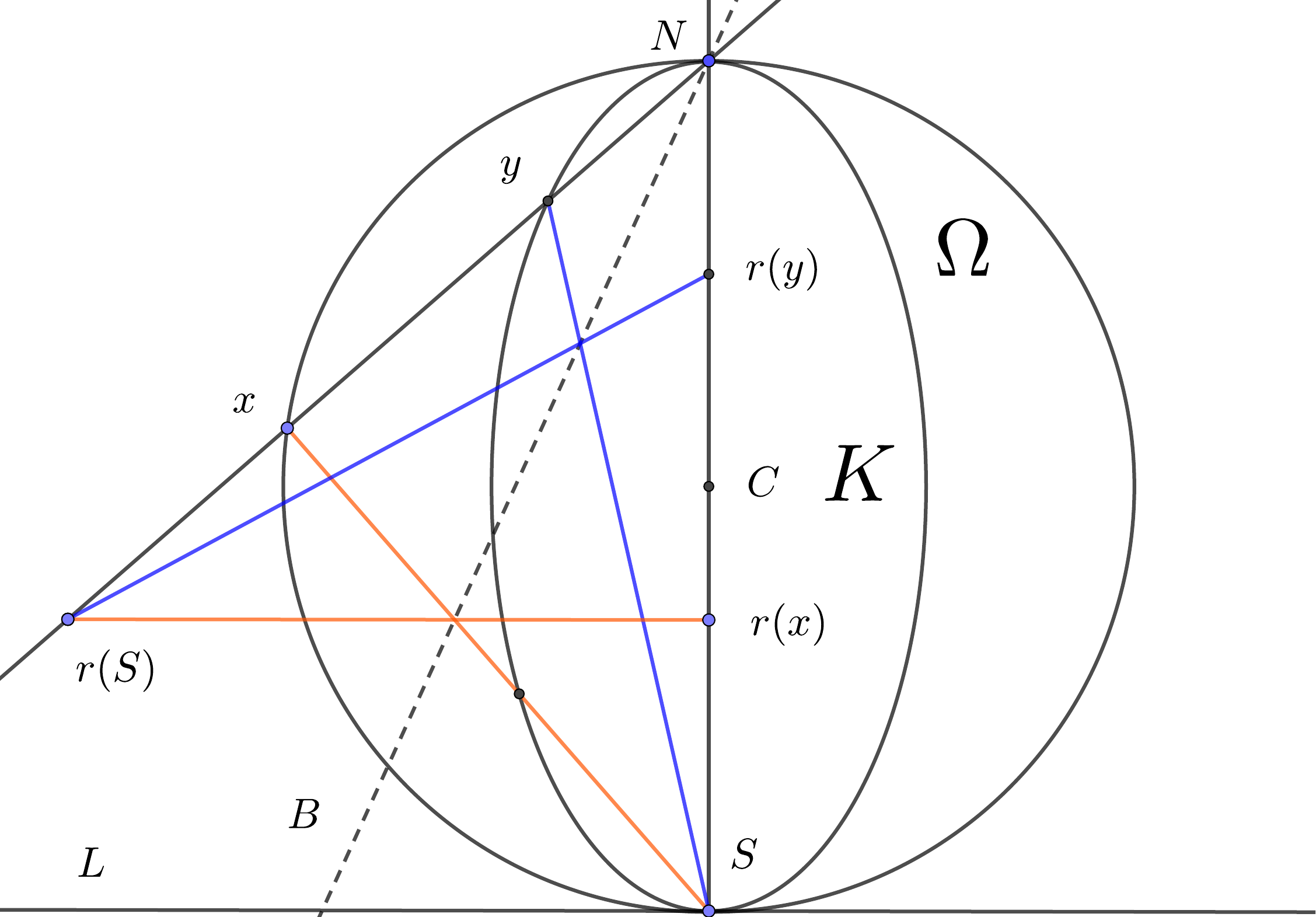}
       \caption{The Stereographic Property characterizes the circle.  Case $N=N', S=S'$.}
        \label{bebe}
     \end{figure}
       \section{Proof of Theorems \ref{proyeeli}, \ref{proyeesfera},  \ref{proyecaracte} and  \ref{teoeli2}}     
 \textbf{Proof of Theorem \ref{proyeeli}.}
       Let $\Gamma$ be a plane such that $\Gamma\cap \inte K\not= \emptyset$ and $N\notin \Gamma$. Let 
       $\Delta$ be a 
       plane perpendicular to $\Gamma$ and containing $L(N,S)$. Since $E$ is a body of revolution every 
       plane containing the 
       axis is a plane of symmetry of $E$. Thus the section $\Gamma \cap E$ has the line $\Delta\cap \Gamma $ 
       as line of symmetry and the cone $C_N(E_{\Gamma})$ has $\Delta $ as plane of symmetry. Let 
       $B$ the bisectriz of the angle $\Delta \cap  C_N(E_{\Gamma})$ and let $H$ be a plane 
       perpendicular to $B$. Since every section of an elliptical cone is an ellipse and since $H$ is a 
       plane perpendicular to $B$, the ellipse $H\cap C_N(E_{\Gamma})$ has the line $\Delta \cap H$ as 
       a line of symmetry. On the other hand, by the choice of $B$, the center of the ellipse 
       $H\cap C_N(E_{\Gamma})$ is in $B$. Therefore the cone $C_N(E_{\Gamma})$ is axially 
       symmetric with axis $B$, i.e., the cone $C_N(E_{\Gamma})$ satisfies the equation (\ref{shorcito}). 
       $\square$
       
       \textbf{Proof of the Theorem \ref{proyeesfera}.}
       Let $\Gamma$ be a plane such that $\Gamma \cap \inte B^3\not=\emptyset$ 
       and let $x\in \mathbb{S}^2 \backslash \Gamma$. Let $y\in \mathbb{S}^2$ 
       such that 
       the line segment $[x,y]$ is a diameter of 
       $\mathbb{S}^2$. By virtue that $\mathbb{S}^2$ is an ellipsoid of revolution 
       with axis the line $L(x,y)$, by Theorem \ref{proyeeli} the cone $C_x(S_{\Gamma})$ is axially symmetric, where $S_{\Gamma}:=\Gamma \cap \mathbb{S}^2$.
              $\square$

              \textbf{Proof of the Theorem \ref{proyecaracte}.}
         To show that the body $K$ is a sphere, we will show that all sections of $K$ are circles.
         Let $\Gamma$ be a plane such that $\Gamma\cap \inte K\not=\emptyset$. To show that the 
         section $K_{\Gamma}$ is a circle, we will show that for every point $x\in \bd K_\Gamma$ there exists a line 
         of symmetry of $K_{\Gamma}$ that passes through $x$.
         
         Let $x\in \bd K_\Gamma$ and let $\{x_n\}\subset \bd K\backslash \Gamma$ be a sequence such that 
         $x_n\rightarrow x$ 
         when $n\rightarrow \infty$. By the hypothesis, for each $n$ there exists a line $L_n$ through 
         $x_n$ such that the rotation $R_{n}^\Gamma$ about $L_n$ by an angle $\pi$ leaves the cone 
         $C_{x_n}(K_\Gamma)$ invariant. Due to the compactness of $\mathbb{S}^2$ and since 
         $x_n\rightarrow x$ there exists a subsequence of $\{L_n\}$, which will again be denoted by 
         $\{L_n\}$, and a line $L$ in $\Gamma$ such that 
         \begin{eqnarray}\label{puro}
         L_n\rightarrow L \textrm{ }\textrm{ and } \textrm{ }x\in L.
         \end{eqnarray}
         We claim that the line $L$ is line of symmetry of $K_{\Gamma}$. Let 
         $y\in \bd K_{\Gamma}\backslash L$ and let 
         \[
         y_n:=R_n^{\Gamma}(y)\in R_n^{\Gamma}(K_{\Gamma}).
         \]
          By (\ref{puro}),
         \[
         R_n^{\Gamma}(K_{\Gamma})\rightarrow K_{\Gamma}.
         \]
         Thus there exists $z\in K_{\Gamma}$ such that $y_n \rightarrow z$. Since we have (\ref{puro}) and, 
         for each $n$, the segment $[y,y_n]$ is perpendicular to $L_n$ and has its midpoint at $L_n$, we 
         conclude that the segment $[y,z]$ is perpendicular to $L$ and has its midpoint at $L$. By the arbitrariness of $y$ it follows that $L$ is line of symmetry of $K_{\Gamma}$. $\square$
       
         
           \textbf{Proof of Theorem \ref{teoeli2}.} We denote by $\Pi_N$ the support plane of $K$ at $N$ and 
           by $L$ the line perpendicular to $\Pi_N$ which pass through $N$.  We claim that $K$ is a body of 
           revolution with axis the line $L$. In order to prove this we are going to prove that every plane 
           containing $L$ is a plane of symmetry of $K$. 
           
           Let $\Delta$ be a plane containing $L$ and let $T\subset \Pi_N$ 
           be a line perpendicular to $\Delta$ passing through $N$. In order to prove that $\Delta$ is a plane of 
           symmetry of $K$ we will prove that for all plane $\Gamma$ such that $\Gamma \cap \inte K\not=\emptyset$ 
           and $T\subset \Gamma$, the section $\Gamma \cap K$ has the line $\Delta \cap \Gamma$ as line of symmetry.

          Let $\Gamma$ be a plane such that $\Gamma\cap \inte K\not=\emptyset$ and $T\subset \Gamma$. 
          Let $\{\Gamma_n\}$ be a sequences of planes such that $\Gamma_n$ is parallel to $T$,   
          $\Gamma_n\cap \inte K\not=\emptyset$  and $\Gamma_n\rightarrow \Gamma$, when $n\rightarrow \infty$. 
          Let $K_n :=\Gamma_n\cap   K$ and let $K_{\Gamma} :=\Gamma \cap   K$. By the hypothesis, for every $n$, 
          there exists a rotation $R_n$ with axis $L_n$ by an angle $\pi$ such that 
          \begin{eqnarray}\label{conchota}
          R_n (C_N(K_n ))=C_N(K_n ).
          \end{eqnarray}
          Due to the compactness of $\mathbb{S}^2$ there exists a subsequence of $\{L_n\}$, which will again be 
          denoted by $\{L_n\}$, and a line $l$ in $\Gamma$ such that 
         \begin{eqnarray}\label{purote}
         L_n\rightarrow l \textrm{ }\textrm{ and } \textrm{ }N\in l.
         \end{eqnarray}       
         Let $x\in \bd K_{\Gamma}\backslash l$. By the choice of $\Gamma_n$ it follows that 
         $K_{n}\rightarrow K_{\Gamma}$. Thus there exists a sequence $\{x_n\}$ such 
         that $x_n\in \bd K_n$ and $x_n\rightarrow x$. 
         Let 
         \[
         y_n:=R_n(x_n)\in R_n(K_n).
         \] 
         Since $K_{n}\rightarrow K_{\Gamma}$ and since (\ref{purote}) holds it follows that 
         $R_n(K_{n})\rightarrow K_{\Gamma}$. Thus there exists $y\in K_{\Gamma}$ such that $y_n \rightarrow y$. 
         Since we have (\ref{purote}) and, for each $n$, the segment $[x_n,y_n]$ is perpendicular to $L_n$ and has 
         its midpoint at $L_n$, we conclude that the segment $[x,y]$ is perpendicular to $l$ and has its midpoint at $l$. 
         By the arbitrariness of $x$ it follows that $l$ is line of symmetry of $K_{\Gamma}$. 
         
         
         Since $N$ is regular point of $K$, the line $T$ is the unique support line of $\Gamma \cap K$ at $N$. Thus $l$ 
         is perpendicular to $T$. It follows that $l=    \Delta \cap \Gamma$. 
          $\square$

              \section{Proof of the Proposition \ref{proyecirculo}}     
       \textbf{Proof of the Proposition \ref{proyecirculo}.}
       Let $\Gamma$ be a plane such that $\Gamma \cap \inte B^3\not=\emptyset$ 
       and let $N\notin \Gamma$.  By Theorem \ref{proyeesfera}, the cone $C_N(S_\Gamma)$ is axially 
       symmetric, where $S_{\Gamma}:=\Gamma \cap \mathbb{S}^2$. Let $L_{\Gamma}$ be axis of the 
        rotation $R^{\Gamma}$ such that $R^{\Gamma}(C_N(S_\Gamma))=C_N(S_\Gamma)$. Let $\Delta$ 
        be a plane of symmetry of $\mathbb{S}^2$ containing $L(N,S)$ and perpendicular to $\Gamma$. 
        Since $\Delta$ is plane of symmetry of the cone $C_N(S_\Gamma)$ the 
        axis $L_{\Gamma}$ is contained in $\Delta$ (otherwise, $C_N(S_\Gamma)$ would have two axis 
        of symmetry which is absurd). Thus $R^{\Gamma}|_{\Delta}$ is a reflection with respect to the 
        bisectriz $B$ of the angle $\Delta \cap C_N(S_\Gamma)$. Since $S_\Gamma$ is a circle, the 
        figure $R^{\Gamma}(S_\Gamma)$ is a circle. Given that $\Delta$ and $\Gamma$ are perpendicular 
        and $\Delta$ is plane of symmetry of $\mathbb{S}^2$, the chord 
        $[a,b]:=(\Delta \cap \Gamma)\cap S_{\Gamma}$ is a diameter of $S_{\Gamma}$. Hence 
        $R^{\Gamma}([a,b])=R^{\Gamma}|_{\Delta}([a,b])$ is the diameter of 
        $R^{\Gamma}(S_{\Gamma})$. By the stereographic property of the circle $\Delta\cap \mathbb{S}^2$ 
        the line segment $R^{\Gamma}|_{\Delta}([a,b])$ is parallel to $\Delta\cap \Pi_S$. Consequently, the 
        homothety, with center of homothety at $N$, which sends $R^{\Gamma}(\Gamma)$ into $\Pi_S$ 
        sends $R^{\Gamma}(S_{\Gamma})$ into $\Psi(S_\Gamma)$. Therefore $\Psi(S_\Gamma)$ is a circle. $\square$

       \section{Proof of Theorem \ref{proyebody}}
        We take a system of coordinates $(x,y,z)$ of $\mathbb{E}^3$ such that the origin is 
        $S$ and the plane $\Pi_S$ is the plane $z=0$. For $u\in \mathbb{S}^2$, we denote by 
        $L(u)$ the line parallel to $u$, $S\in L(u)$, and by $u^{\perp}$ the plane perpendicular to $u$, 
        $S\in u^{\perp}$. 
        
        \begin{lemma}\label{prerevolution}
        Let $u\in (\mathbb{S}^2\cap \Pi_S)$ and let $\Gamma$ be a plane such that 
        $\Gamma\cap \inte K\not= \emptyset$, $N\notin \Gamma$ and $\Gamma$ parallel to $L(u)$ 
        and no parallel to $\Pi_S$. Then the axis $L_{\Gamma}$ of the rotation $R^{\Gamma}$ which satisfies 
        (\ref{zapatito}) and (\ref{playerita}) is contained in the plane $N+u^{\perp}$.
        \end{lemma}
        \begin{proof}
        The plane of $K'_{\Gamma}$ is parallel to $\Gamma$. Thus $K'_{\Gamma}$ is parallel to $L(u)$. 
         Consequently $I:=\Psi(K_{\Gamma})\cap K'_{\Gamma}$ is a line segment parallel to $L(u)$. By 
         (\ref{playerita}), $R^{\Gamma}(I)=I$ and from here it follows that $L_{\Gamma}$ is perpendicular to 
         $I$ (only the lines perpendicular to $L_{\Gamma}$ and which have non-empty intersection with 
         $L_{\Gamma}$ are invariant of $R^{\Gamma}$), i.e., $L_{\Gamma}$ is perpendicular to $L(u)$. 
         Therefore $L_{\Gamma}\subset (N+u^{\perp})$. 
        \end{proof}

       \begin{lemma}\label{revolution}
        The body $K$ is a body of revolution with axis the line $L(N,S)$.
        \end{lemma}
        \begin{proof} 
        In order to prove that $K$ is a body of revolution with axis the line $L(N,S)$ we are going to prove that every plane 
        containing $L(N,S)$ is a 
       plane of symmetry of $K$. Let $\Delta$ be a plane containing $L(N,S)$ and let $x\in \bd K\backslash \Delta$. Let 
       $y\in \bd K$ such that the line $L(x,y)$ is perpendicular to $\Delta$, let $z:=L(x,y)\cap \Delta$ and let 
       $a,b\in \bd (\Delta \cap K)$ such that the point $z$ belongs to the bisectriz $B$ of the angle defined by the lines 
       $L(a,N), L(b, N)\subset \Delta$. We denote by $\Gamma$ the plane defined by the points $a,b,x,y$ and let $u$ be a 
       unit vector perpendicular to $\Delta$. By virtue of the hypothesis, there exists a rotation $R^{\Gamma}$ such the 
       relation (\ref{zapatito}) holds. By Lemma \ref{prerevolution}, 
       $L_{\Gamma}\subset (N+u^{\perp})$ but $\Delta=N+u^{\perp}$. Hence  $R^{\Gamma}(\Delta)=\Delta$ and  
       \[
       R^{\Gamma}(L(a,N))=L(b,N)\textrm{ and }R^{\Gamma}(L(b,N))=L(a,N).
       \] 
       It follows that $B=L_{\Gamma}$. Thus 
       \[
       R^{\Gamma}(L(x,N))=L(y,N), 
       \] 
        i.e., The length of  $[x,z]$ is equal to the length of $[z,y]$. Therefore $\Delta$ is a plane of symmetry of $K$.
      \end{proof}
 
       \begin{lemma}\label{seccionestereo}
       Let $\Delta$ be a plane containing the line $L(N,S)$. Then $\Delta \cap K$ is a circle.
       \end{lemma}
       \begin{proof}
       Let $\Delta$ be a plane containing the line $L(N,S)$. In order  to prove that $\Delta \cap K$ is a circle we will use 
       Lemma \ref{estereo}, thus we are going to prove that  $\Delta \cap K$ satisfies the stereographic property. Let 
       $u\in (\mathbb{S}^2\cap \Pi_S)$ perpendicular to $\Delta$, let $a,b\in \Delta \cap K$ and let $\Gamma$ 
       be a plane containing $L(a,b)$ and parallel to $u$. We denote by $B$  the bisectriz of the angle defined by the lines $L(a,N),L(b,N)\subset \Delta$
       and by $r(a),r(b)$ the reflections of $a,b$ with respect to $B$. We are going to prove that the line segment 
        $r(a),r(b)$ is parallel to  $ \Pi_S\cap \Delta$.
       
       By virtue of the hypothesis, there exists a rotation 
       $R^{\Gamma}$ such the relations (\ref{zapatito}) and (\ref{playerita}) holds, where $K'_{\Gamma}$ is a homothetic copy of $K_{\Gamma}$ by an homothety with center 
        at $N$. By Lemma \ref{prerevolution}, 
       $L_{\Gamma}\subset (N+u^{\perp})$ but $\Delta=N+u^{\perp}$. Hence  $R^{\Gamma}(\Delta)=\Delta$ and  
       \[
       R^{\Gamma}(L(a,N))=L(b,N)\textrm{ and }R^{\Gamma}(L(b,N))=L(a,N).
       \] 
       It follows that $B=L_{\Gamma}$ and $R^{\Gamma}|_{\Delta}$ is the reflection with respect to $B$. 
       
       Let $K'_{\Gamma} \cap \Delta:=[a',b']$. By (\ref{playerita}) 
       \[
       \Psi (K_{\Gamma})\cap \Delta =R^{\Gamma}(K'_{\Gamma}) \cap \Delta.
       \]
        On the one hand,
         \[
         \Psi ([a,b])= \Psi (K_{\Gamma} \cap \Delta)=\Psi (K_{\Gamma})\cap \Delta
         \]
         and, on the other hand,
         \[
        R^{\Gamma}(K'_{\Gamma}) \cap \Delta=R^{\Gamma}(K'_{\Gamma} \cap \Delta)=R^{\Gamma}([a',b'])
         \]         
       Consequently      
        \begin{eqnarray}\label{melapelan}
        \Psi ([a,b])= R^{\Gamma}([a',b'])
        \end{eqnarray}
        Since $K'_{\Gamma}$ is a homothetic copy of $K_{\Gamma}$ by an homothety with center 
        at $N$ the line segment $[a',b']$ is parallel to $[a,b]$. Hence $R^{\Gamma}([a',b'])=R^{\Gamma}|_{\Delta}([a',b'])$ is parallel to 
        $r(a),r(b)$ (the reflection with respect to $B$,  in the plane $\Delta$,  preserves parallelism). By (\ref{melapelan}), the line segment 
        $r(a),r(b)$ is parallel to  $\Psi ([a,b])\subset \Pi_S\cap \Delta$.
         \end{proof}
         \textbf{Proof of Theorem \ref{proyebody}.}
         The proof of the Theorem \ref{proyebody} follows immediately from Lemmas \ref{revolution} and \ref{seccionestereo}. 
       
           \textbf{Author Contributions.} Material preparation were performed by Efr\'en Morales 
           Amaya. The first draft of the manuscript was written by Efr\'n Morales Amaya.  
           
          \textbf{Funding.} The authors declare that no funds, grants, or other support were received 
          during the preparation of this manuscript.

          \textbf{Data Availability.} Data sharing not applicable to this article as no datasets were 
          generated or analyzed during  the current study.
          
          \textbf{Declarations.}
          
           \textbf{Conflict of interest.} The authors have no relevant financial or non-financial interests to disclose.

\end{document}